\documentclass[12pt]{amsart}

\usepackage{amssymb}

\usepackage{enumerate}

\usepackage{graphicx}

\makeatletter
\@namedef{subjclassname@2010}{%
  \textup{2010} Mathematics Subject Classification}
\makeatother

\newtheorem{thm}[equation]{Theorem}
\newtheorem{corollary}[equation]{Corollary}
\newtheorem{lemma}[equation]{Lemma}
\newtheorem{proposition}[equation]{Proposition}

\theoremstyle{definition}

\numberwithin{equation}{section}

\title[Monomial Hilbert Transform] {Sparse Bounds for  Maximal Monomial \\ Oscillatory  Hilbert Transforms} 

\author{Ben Krause}
\address{
Department of Mathematics
The University of British Columbia \\
1984 Mathematics Road
Vancouver, B.C.
Canada V6T 1Z2}
\email{benkrause@math.ubc.ca}
\thanks{Research supported in part by  an NSF Postdoctoral Research Fellowship.}

\author{Michael T. Lacey}   

\address{ School of Mathematics, Georgia Institute of Technology, Atlanta GA 3034, USA}
\email {lacey@math.gatech.edu}
\thanks{Research supported in part by grant NSF-DMS 1265570 and  NSF-DMS-1600693.}

\usepackage{paralist} 
\usepackage{amsrefs}
\begin{document}

\begin{abstract}
For each $ d \geq 2$, the  Hilbert transform with a polynomial oscillation as below  satisfies a $ (1, r )$ sparse bound, 
for all  $ r>1$
\begin{equation}
H _{ \ast  } f (x) = 
\sup _{\epsilon } \Bigl\lvert 
\int_{|y| > \epsilon}  f (x-y) \frac { e ^{2 \pi i  y ^d }} y\; dy  
\Bigr\rvert. 
\end{equation}
This quickly implies   weak-type inequalities for the maximal truncations,  which hold for $A_1$ weights, but are new even in the case of Lebesgue measure. 
The  unweighted weak-type estimate  \emph{without maximal truncations} but with arbitrary polynomials, is due to Chanillo and Christ (1987).  
\end{abstract} 

\maketitle

\section{Introduction}  

The theory of oscillatory singular integrals, initiated by Ricci and Stein \cites{MR822187,MR890662}, 
 concerns operators of the form 
\begin{equation} \label{e:TK}
T_P f (x) = \int e ^{i P  (x,y)}K (y) f (x-y) \; dy . 
\end{equation}
where $ K (y)$ is a Calder\'on-Zygmund kernel on $ \mathbb R ^{n}$, and $ P (x,y)$ is a polynomial of two variables.  
At this stage 
the $ L ^{p}$  theory of the same is advanced \cites {MR1879821}.  (Also see \cites{MR2545246,11054504}.) The $ L ^{1}$ theory is harder, with the dominant result being that of 
Chanillo and Christ  \cite{MR883667}.   Combining  \cites{MR890662,MR883667}, we have

\begin{thm}\label{t:fixedPoly}  For $ 1 < p < \infty $,  the operator $ T_P$ is bounded on $ L ^{p}$, that is 
\begin{equation*} 
\lVert T_P  \,:\,  L^{p} \mapsto L ^{p}\rVert  \lesssim  1, 
\end{equation*}
where the implied constant depends on the degree of $ P$, and in particular is independent of $ \lambda $.  
Moreover, $ T_P$  maps $ L ^{1}$ to weak $ L ^{1}$, with the same bound.  
\end{thm}

It is very easy to extend the $ L ^{p}$ result above to maximal truncations, but the weak-type bounds for maximal truncations are unknown.  
We address the maximal truncations at the $ L ^{1}$ endpoint, in what is the simplest possible case, namely one dimension, with the Hilbert transform kernel, and oscillatory terms that are just monomials.  
Even in this restricted setting, our results are new.  

But moreover, we use the recent language of sparse forms to provide  quantitative bounds, which immediately provide new weighted inequalities.    Call a collection of cubes $ \mathcal S$  in $ \mathbb R ^{n}$ \emph{sparse} if there 
are sets $ \{ E_S  \,:\, S\in \mathcal S\}$  
which are pairwise disjoint,   $E_S\subset  S$ and satisfy $ \lvert  E_S\rvert > \tfrac 14 \lvert  S\rvert  $ for all $ S\in \mathcal S$.
For any cube $ I$ and $ 1\leq r < \infty $, set $ \langle f \rangle_ {I,r} ^{r} = \lvert  I\rvert ^{-1} \int _{I} \lvert  f\rvert ^{r}\; dx  $.  Then the $ (r,s)$-sparse form $ \Lambda _{\mathcal S, r,s} = \Lambda _{r,s} $, indexed by the sparse collection $ \mathcal S$ is 
\begin{equation*}
\Lambda _{S, r, s} (f,g) = \sum_{I\in \mathcal S} \lvert  I\rvert \langle f  \rangle _{I,r} \langle g \rangle _{I,s}.  
\end{equation*}
Given a  sublinear operator $ T$, and $ 1\leq r, s < 1$, we set 
$ \lVert T \,:\, (r,s)\rVert$ to be the infimum over constants $ C$ so that for all  all bounded compactly supported functions $ f, g$, 
\begin{equation}\label{e:SF}
\lvert  \langle T f, g \rangle \rvert \leq C \sup  \Lambda _{r,s} (f,g), 
\end{equation}
where the supremum is over all sparse forms.  
It is essential that the sparse form be allowed to depend upon $ f $ and $ g$. But the point is that the sparse form itself varies over a class of operators with very nice properties.   

For singular integrals without oscillatory terms we have 

\begin{thm}\label{t:Tsparse} \cites{14094351,MR3625108}  Let $ K $ be a Calder\'on-Zygmund kernel on $ \mathbb R ^{n}$ as above. Then, the  operator  $ T f = \textup{p.v.} K \ast f (x)$  
satisfies $ \lVert T \,:\, (1,1)\rVert < \infty $.  
\end{thm}

The interest in this result is that the (1,1) sparse bound implies virtually all the known norm bounds for a Calder\'on-Zygmund operator on a lattice, including  weighted $ L ^{p}$ and weak-type estimates, with sharp dependence upon $ p$ and the $ A_p$ characteristic of the weight.  

Surprisingly,  a very easy proof  by Spencer and one of us  provides  sparse bounds for the polynomial case. 

\begin{thm}\label{t:rr} \cite{MR3611077}  Assume that the polynomial $ P (x,y)$ is only a function of $ y$ in \eqref{e:TK}. 
Then, for all $ 1< r < \infty $,   we have $ \lVert T_P \,:\, (r,r)\rVert < \infty $.  
\end{thm}

This result is strong enough to deduce $ A_p $ weighted inequalities, for all $ 1< p < \infty $, and trivally extends to maximal truncations.  Compare to \cite{MR2900003}. 

We prove this sparse bound: A  $ (1, r)$ bound for maximal truncations of an oscillatory Hilbert transform.  (The case of degree one is excluded below, since it falls within the scope of Theorem~\ref{t:Tsparse}.) 

\begin{thm}\label{t:monomial} Let $ d \geq 2$, and define 
\begin{align}  
H _{\ast, d }  f & := \sup _{\epsilon >0} 
\Bigl\lvert 
\int _{\lvert  y\rvert > \epsilon  }  f (x-y) \frac { e ^{2 \pi i  y ^d }} y\; dy 
\Bigr\rvert
\end{align}
This  operator satisfy a $ (1,r)$ bounds, for $ 1< r \leq 2$.  Namely, 
\begin{equation}  \label{e:H<S}
\lVert H_{ \ast ,d} \,:\,  (1,r)\rVert \lesssim \frac {1} {(r-1) }. 
\end{equation}
\end{thm}

The  sparse forms are positive, and highly localized, making their properties on weighted spaces very easy to analyze.  
We have as an easy consequence a range of quantitative weighted inequalities for  $ H _{\ast ,d}$, phrased in the language of Muckenhoupt $A_p$ weights. 

\begin{corollary}\label{c:wtd} For every  $ d\geq 2$ and weight $ w\in A_1$ there holds 
\begin{gather*}
\lVert  H _{\ast, d  }  \,:\, L ^{1} (w) \mapsto L ^{1, \infty } (w)\rVert \lesssim [w] _{A_1} ^{2} \log_+ [w] _{A_1}, 
\\
\lVert  H _{\ast, d  }  \,:\, L ^{p} (w) \mapsto L ^{p}  (w)\rVert \lesssim [w] _{A_p}  ^{ \max \{\frac 2 {p-1}, \frac p {p-1}\}}, \qquad 1< p < \infty . 
\end{gather*}
\end{corollary}


For the second bound, see \cite{MR3531367}*{\S 6}, or the proof of \cite{CCPO}*{Cor. A.1}.  
The quantitative bound in $L^p(w)$ is new, with prior work \cites{MR1782909,MR2900003} being rather complicated, while also addressing more general situations.  The weak type bound can be found in \cite{170105170}*{Thm 1.11}. 
The weak-type estimate is new for maximal truncations  \emph{even in the unweighted case.}   
Without maximal truncations, in the unweighted case,  the weak $ L ^{1}$ is a well known result of Chanillo and Christ \cite{MR883667} from 1987.  
The recent work of \cite{CCPO} includes techniques powerful enough to prove the sparse bound \eqref{e:H<S} 
for $ H_d$.  (Follow the lines of their paper for their Theorem C, a sparse bound for  the Bochner-Riesz multiplier at critical index.) But that paper does not address maximal truncations.

\smallskip

Sparse bounds for operators arose from the weighted theory, particularly motivated by the work of Andrei Lerner \cites{MR2721744,MR3085756}. 
The bilinear form estimate was proved first by Cond\'e and Rey \cite{14094351}, with the subsequent proof of one of us \cite{MR3625108} having several interesting extensions, see for instance \cites{MR3531367,160506401}. 
This paper is strongly motivated by the multilinear approach of Culiuc, Ou and Di Plinio \cite{160305317},  the `rough singular integral' paper of Cond\'e, Culiuc, Di Plinio, and Ou \cite{CCPO}, and a paper by  Spencer and one of us \cite{MR3611077}.  This last paper conjectured therein that a $ (1,r)$ bound held in the generality of Theorem~\ref{t:fixedPoly}.   This paper validates that conjecture, and indicates that a significantly stronger result is true.

\medskip 
As we mentioned, the recent work \cite{CCPO} of Cond\'e,   Culiuc,  Di Plinio and  Ou, proves the sparse bound above without maximal truncations. 
This paper supplements their analysis with a technique to control maximal truncations.  

\begin{inparaenum}
\item Following a standard technique in the subject \cites{MR883667,MR1879821,MR822187}, the proof is based upon certain $ T  T ^{\ast} $ calculations.  The latter are summarized in Lemma~\ref{l:same}, and are quite simple in the monomial case.   

\item  The operator in question has to be `localized', and the main estimate is Lemma~\ref{l:sparse}, which shows that if the operator acts on that part of a function that has bounded averages, then the operator is highly integrable, against functions which have controlled averages.   
The deduction of the theorem from this statement is a standard recursion.

\item The proof of the crucial lemma uses the methods of Christ \cites{MR951506,MR796439} and Chanillo and Christ \cite{MR883667}. 
And, as we indicated, one could use the general procedure of \cite{CCPO} to complete the $(1,r)$ bound, but without maximal truncations.  
We introduce  one more technique,  based around a Carleson measure estimate in \eqref{e:CM}, with an  abstract Rademacher-Menshov theorem 
to control  maximal truncations.  
\end{inparaenum}

\medskip 
\textbf{Acknowledgment.}
We thank Jos\'e Cond\'e, Amalia Culiuc,  Francesco Di Plinio and Yumeng Ou for discussions, and sharing their paper \cite{CCPO}.  

\section{Lemmas} \label{s:lemmas}

There are two categories of facts collected here,  
\begin{inparaenum}
\item those which reflect the oscillatory nature of the problem, 
\item and,  a variant of the Rademacher-Menshov theorem, which will control the maximal truncations. 
\end{inparaenum}

\subsection{Notation}\label{not}
Henceforth, we use $e(t) := e^{2\pi i t}$; $M_{HL}$ denotes the Hardy-Littlewood maximal function.

Let $|\rho| \lesssim 1_{|t| \approx 1}$ be an odd compactly supported Schwartz function that resolves the singularity $\frac{1}{t}$ in that
\begin{equation}\label{e:rho}
 \sum_i \rho_i(t) := \sum_i 2^{-i} \rho(2^{-i}t) = \frac{1}{t}, \qquad  t\neq 0.
\end{equation}
Let
\begin{equation}\label{e:psi}
\rho ^{+}_i(t) := \rho_i (t) \mathbf{1}_{t > 0};
\end{equation}
The oscillatory part of the argument concentrates on  $\{ \rho ^{+}_i \}$, with the understanding that symmetric arguments can be used to treat $\{ \rho_i - \rho^+_i \}$.

We will make use of the modified Vinogradov notation. We use $X \lesssim Y$, or $Y \gtrsim X$ to denote the estimate $X \leq CY$ for an absolute constant $C$. 
We use $X \approx Y$ as
shorthand for $Y \lesssim X \lesssim Y $.

\subsection{Oscillatory Estimates}
We will be concerned with operators that are convolution with respect to $ \psi _k (y) :=  e (y ^{d} ) \rho  ^{+} _k (y)$ for $ k \in \mathbb N $.  
These next two lemmas are essential facts about these operators.  

\begin{lemma}\label{l:same}  There is  a choice of $ k_0 >0$  so that for all  $ k \in \mathbb N $, with $ k > k_0$, we have 
\begin{equation}\label{e:same}
\lvert  \tilde \psi_k \ast \psi_k (x) \rvert 
\lesssim 2 ^{-k} \mathbf 1_{ [- 1/2 , 1/2]} (x)+      2 ^{-2k}  \mathbf 1_{ [-2 ^{k+1} , 2 ^{k+1})}(x).  
\end{equation}
Above, $ \tilde \phi (y) = \overline  \phi (-y)$. 
\end{lemma}

\begin{proof}
The convolution is explicitly 
\begin{equation*}
  \tilde \psi_k \ast \psi _k (x) = 
  \int e (  (x+y)^d - y ^{d} )  \rho ^{+} _k (x+y) \rho ^{+}_k (y) \; dy .
\end{equation*}
For $|x| \leq 1/2$, we use the trivial bound on the integral of $2^{-k}$, so we consider the remaining case, when $1/2 \leq |x| \leq 2^{k+1}$.

We first address the case of $ d \geq 3$.   The derivative of the phase is
\[ \lvert  d(x+y)^{d-1} - d y^{d-1}\rvert \simeq 
\lvert  x 2 ^{k (d-2)}\rvert 
\gtrsim  2^{k(d-2)}, \qquad  \lvert  x\rvert\geq 1/2 .\]
A simple integration by parts argument allows us to estimate to conclude the estimate. 

For the case of $ d=2$, we should bound the integral 
\begin{equation*}
\int e (2xy)\rho ^{+} _k (x+y) \rho ^{+}_k (y) \; dy  . 
\end{equation*}
This is the Fourier transform of the Schwartz function $ \rho ^{+} _k (x+ \cdot ) \rho ^{+}_k ( \cdot ) $, evaluated at $ \lvert  x\rvert \geq 1/2 $.  The latter function has spatial scale $ 2 ^{k}$,  so the Fourier decay is on scale $ 2 ^{-k}$, and the bound follows.

\end{proof}

We now prove the orthogonality statement. 

\begin{lemma}\label{l:different}  For  an absolute constant $ k_0 >0$, and 
$ j, k \in \mathbb N $, with $ 1\leq j < k - k_0   $, we have 
\begin{equation}\label{e:different}
\lvert  \tilde \psi_j \ast \psi _k(x) \rvert 
\lesssim  2 ^{-2k   }    \mathbf 1_{ [-2 ^{k+1} , 2 ^{k+1})}.
\end{equation}

\end{lemma}

\begin{proof}
The  convolution is 
\begin{equation*}
  \tilde \psi_j \ast \psi_k (x) = 
  \int e (  (x+y) ^{d} - y ^{d} ) \rho  ^{+} _k (x+y) \rho ^{+}_j  (y) \; dy 
\end{equation*}

Since $(x+y)$ is so much larger than $y$, the derivative of the phase is $\approx 2^{k(d-1)}$, so the result follows by a simple integration by parts when $d \geq 3$. When $d = 2$, we notice that the second derivative (in $y$) of the phase vanishes, and integrate by parts twice.
\end{proof}

\subsection{Rademacher-Menshov Theorem}

We need a general principle to convert orthogonality inequalities into bounds  for maximal truncations. 
Namely, we need a  variant of the  Rademacher-Menshov inequality.   This has been observed many times, for an explicit formulation and proof, see  \cite{MR2403711}*{Thm 10.6}. 

\begin{lemma}\label{l:RM}  Let $ (X, \mu )$ be a measure space, and $ \{ \phi _j \,:\, 1\leq j \leq N\}$ a sequence of functions 
which satisfy the Bessel type inequality below, for all sequences of coefficients $c_j \in \{ 0, \pm 1\}$, 
\begin{equation}\label{e:bessel}
\Bigl\lVert \sum_{j=1} ^{N}  c_j \phi _j \Bigr\rVert _{L ^2 (X)} \leq A .  
\end{equation}
Then, there holds 
\begin{equation}\label{e:RM}
\Bigl\lVert\sup _{1< n \leq N} 
\Bigl\lvert 
 \sum_{j=1} ^{n}   \phi _j
\Bigr\rvert
\Bigr\rVert _{L ^2 (X)} \lesssim A   \log(2+ N) .  
\end{equation}

\end{lemma}

\section{Proof of the Main Theorem} \label{s:proof}
Let $k_0 \lesssim 1$ be as in Lemma~\ref{l:same} and Lemma~\ref{l:different}, and recall the notation in \eqref{e:rho}.  
The operator
\[ \aligned 
f &\mapsto  \sum_{j < k_0} \int e (y ^{d}) f (x-y)   \rho_j(y) \; dy \\
& \qquad = \sum_{j < 0} \int e (y ^{d}) f (x-y)   \rho_j(y) \; dy + \sum_{j=0}^{k_0} \int e (y ^{d}) f (x-y)   \rho_j(y) \; dy 
\endaligned\]
is the sum of two operators. The first is a Calder\'on-Zygmund operator, with Calder\'on-Zygmund norm bounded independently of $d$; its maximal truncations  satisfy the better $ (1,1)$ sparse bound as stated in Theorem~\ref{t:Tsparse}. 
The second is just bounded by a multiple of the Hardy-Littlewood maximal function, which is well known to satisfy a sparse $(1,1)$ bound.

It suffices to consider the complementary operator. For it, the fact that $ \rho _j$ integrates to zero  is not relevant, and it suffices to consider the operator 
\begin{equation} \label{e:M0}
 \sum_{k = k_0} ^{\infty } \int e (y ^{d}) \rho ^{+}_k (y) f (x-y) \; dy 
\end{equation}
where $ \rho ^{+}_k $ is as in \eqref{e:psi}. 
The maximal truncations of the operator in \eqref{e:M0} are dominated by 
\begin{equation*}
\sup _{l \geq k_0} 
\Bigl\lvert 
\sum_{k=l} ^{\infty } \int e (y ^{d}) \rho ^{+} _k (y) f (x-y) \; dy 
\Bigr\rvert +M _{HL} f =: T_{\ast } f + M_{HL}f.
\end{equation*}
 In analyzing $ T _{\ast } f$, there is no additional cancellation properties of $ f$ needed, and so we assume that $ f $ is non-negative, for simplicity below. 

We make a dyadic reduction, to facilitate the various localizations we will need. 
Recall that there are three \emph{shifted dyadic grids}  $ \mathcal D_j$, for $ j=1,2,3$, so that for each $k \in \mathbb{Z}$
\begin{equation*}
\sum_{j=1} ^{3} \sum_{|I| = 2^k, I \in \mathcal{D}_j} \mathbf 1_{I'} \equiv 1, \qquad I' = \tfrac 13I.  
\end{equation*}
For an interval, set 
\begin{equation} \label{e:TI}
T _{I} g (x) = \int   \psi _k (y)  (g \mathbf 1_{I'} ) (x-y)\; dy, \qquad \lvert  I\rvert = 2^{k+2}.   
\end{equation}
Recall that $ \psi _k (y) = e ( y ^{d}) \rho ^{+} _k (y)$ incorporates the oscillatory term into the kernel.  
With this choice $ T_I g $ is supported on $ I$.  Define, for a collection of intervals $ \mathcal I$, 
\begin{gather} \label{e:TD}
T _{\mathcal I} f := 
 \sum_{I \in  \mathcal I } T_I f 
, \qquad j =1,2,3.  
\\
\label{e:dyadic}
T _{\ast , \mathcal I} f := \sup _{l \geq k_0} 
\Bigl\lvert 
 \sum_{I \in \mathcal I \,:\, \lvert  I\rvert \geq 2 ^{l} } T_I f 
\Bigr\rvert, 
\end{gather}

It suffices to show the claimed sparse bounds for  $ T _{\ast ,\mathcal D_j ^{+}} $, for $ j=1,2,3$, where 
\begin{equation*}
\mathcal D_j ^{+} = \{I\in \mathcal D_j \,:\, \lvert  I\rvert \geq 2 ^{k_0} \}. 
\end{equation*}
There is no additional property of the shifted dyadic grids   used, so we suppress the subscript in  $  \mathcal D_j ^{+}$ below.

It is well known that  $ \lVert T _{\ast , \mathcal I} f \rVert_2 \lesssim \lVert f\rVert_2 $, for any subset $ \mathcal I \subset \mathcal D ^{+}$.  The main Lemma is an $ L^1 \to L ^q$ refinement of this inequality.

\begin{lemma}\label{l:sparse} Let $ K >4$ be a fixed constant.   For any interval $ I_0 \in \mathcal D$ and collection of subintervals $ \mathcal I$ of $ I_0$, provided 
\begin{equation}\label{e:K}
\sup _{I\in \mathcal I} \langle f \rangle_I \leq K \langle f \rangle _{I_0}, 
\qquad \sup _{I\in \mathcal I} \langle g \rangle_I \leq K \langle g \rangle _{I_0}, 
\end{equation}
we have the inequalities below, holding uniformly in $ 2\leq q < \infty $. 
\begin{align} \label{e:q*}
\langle  T _{\ast , \mathcal I} f, g\rangle & \lesssim 
q    \lvert  I_0\rvert \langle f \rangle _{I_0}  \langle g\rangle _{I_0,q'}. 
\end{align}
\end{lemma}

\begin{proof}[Proof of Theorem~\ref{t:monomial}, assuming Lemma~\ref{l:sparse}]
It suffices to show that for bounded compactly supported $ f, g$, we have for $2 \leq q < \infty$, 
\begin{equation} \label{e:qj}
\langle  T _{\ast , \mathcal D ^{+}}   f, g\rangle  \lesssim  q   \sum_{S\in \mathcal S} \lvert  S\rvert \langle    f  \rangle_S \langle    g   \rangle _{S, q' }, 
\end{equation}
where $ T _{\ast , \mathcal D ^{+}}  $ is defined in \eqref{e:TD}, and some choice of sparse collection $ \mathcal S$.  

We can assume that  non-negative $ f , g $ are supported on a  dyadic  interval  $ I_0 $.  Note that `above $I_0$' we have 
\begin{equation}
\mathbf{1}_{I_0} 
 \sum_{J \,:\, J\supset I_0} 
\lvert T_J f \rvert \lesssim \langle f \rangle _{I_0}, 
\end{equation}
Make the interval $ I_0$   the maximal element of the  sparse collection $ \mathcal S$.  
From this, it suffices to restrict the sum intervals $ I\subset I_0$. 
We take $ \mathcal I_0 = \{I \,:\,  I\subset I_0\}$,  further set $ \mathcal E $ to be the maximal subintervals $ K\subset I_0$ such that $ \langle f \rangle_K > 10 \langle f \rangle _{I_0}$ and/or $ \langle g \rangle_K > 10 \langle g \rangle _{I_0}$.  Then, the set $ E = \bigcup _{K\in \mathcal E} K$ 
has measure at most $ \tfrac 1 {5} \lvert  I_0\rvert $.   Let $ \mathcal I = \{I\in \mathcal I_0 \,:\, I\not\subset E\}$, 
and set $ \mathcal I_0 (K) := \{I\in \mathcal I_0 \,:\, I\subset K\}$. 
We have 
\begin{align*}
\lvert  \langle T _{\ast ,\mathcal I_0} f_1,f_2 \rangle\rvert  
&\leq \lvert  \langle T _{\ast , \mathcal I} f_1,f_2 \rangle\rvert  
+ \sum_{K\in \mathcal E_0} 
\lvert  \langle T _{\ast ,\mathcal I_{0} (K)} f_1,f_2 \rangle\rvert 
\end{align*}
But the first term on the right is  bounded by \eqref{e:q*}, namely 
\begin{equation*}
q \langle f \rangle _{I_0} \langle g \rangle _{I_0, q' } \lvert  I_0\rvert , 
\end{equation*}
And we add the collection $ \mathcal E$ to the sparse collection $ \mathcal S$, and then recurse.  

\end{proof}

\section{Proof of Lemma~\ref{l:sparse}} 
As a matter of simplicity, we assume that $ \langle f  \rangle _{I_0} =1 $.  The majority of the analysis will be done on $f$, with the  assumption \eqref{e:K}  used on the second function $g$ at just one point.
Then,  we take $ \mathcal B$ to be the maximal subintervals $ J$ of $ I_0$ 
so that 
$
 \langle f \rangle _{J} > K 
$. 
Then, write a `good-bad' decomposition of $ f = \gamma +b $ where $ b = \sum_{J\in \mathcal B} f \mathbf 1_{J}$. 
(This is different from the typical Calder\'on-Zygmund decomposition, since  cancellative properties of $ b$ can not be used in the oscillatory context!)   

The good function is easy to dispense with. 
\begin{proposition}\label{p:good}  We have 
\begin{align*}
\lVert T _{\ast , \mathcal I} \gamma \rVert_q 
& \leq \lVert T _{\ast , \mathcal I} \,:\, L ^q \mapsto L ^q \rVert  \cdot  \lVert \gamma \rVert_q \lesssim q \lvert  I_0\rvert ^{1/q}.  
\end{align*}

\end{proposition}

\begin{proof}
This depends upon the $ L ^{q}$ norm estimate, which is easy.  
Let $ \mathcal I _0 (k) := \{I\in \mathcal I \,:\, \lvert  I\rvert = 2 ^{k}\}$. We have 
\begin{gather*}
\lVert T _{\ast , \mathcal I (k)} \,:\, L ^2 \mapsto L ^2 \rVert \lesssim 2 ^{-k/2}, 
\qquad 
\lVert T _{\ast , \mathcal I (k)} \,:\, L ^ \infty  \mapsto L ^ \infty  \rVert \lesssim 1. 
\end{gather*}
The first follows from the oscillatory estimate \eqref{e:same}, while the second is trivial. Interpolating, we have for $2\leq q < \infty$,
\begin{align*}
\lVert T _{\ast , \mathcal I} \,:\, L ^q  \to L ^q \rVert   
&\lesssim \sum_{k=k_d} ^{\infty } \lVert T _{\ast , \mathcal I (k)} \,:\, L ^q \mapsto L ^q \rVert
\lesssim \sum_{k=k_d} ^{\infty }  2 ^{-k/2q} \lesssim q. 
\end{align*}

\end{proof}

Thus, it remains to consider the `bad' function.  Some additional notations are required before we arrive at the core of the argument.  Let $ \mathcal B (k) = \{J\in \mathcal B \,:\,  \max\{ 2 ^{k_0}, \lvert  J\rvert \} = 2 ^{k} \} $, for $k\geq k_0$,  and 
\begin{equation}\label{e:B}
b = \sum_{k=k_0} ^{\infty } \sum_{J\in \mathcal B (k)} f \mathbf 1_{J} =:   \sum_{k=k_0} ^{\infty } B_k .  
\end{equation}
Note that if $ I\in \mathcal I$ and $ J\in \mathcal B$, we must have $ I\cap J \in \{\emptyset , J\}$. 
(That is, the `bad' interval must be smaller.) 
Therefore, we have 
\begin{align*}
\sum_{I\in \mathcal I} T_I b 
= \sum_{s= 0} ^{\infty }  \sum_{j = k_0} ^{\infty } T _{\mathcal I  (j)} B _{j-s}, 
\qquad \mathcal I (k) = \{I\in \mathcal I \,:\, \lvert  I\rvert = 2 ^{k} \}. 
\end{align*}
The crucial facts of the $ \{B_j\}$ are 
\begin{gather} \label{e:Bsum}
\sum_{j} \lVert B_j\rVert_1 \leq \lvert  I_0\rvert, 
\\
\label{e:BI}
\int _{K} B_j (y) \; dx \lesssim  2 ^{j}, \text{ for any } |K| \lesssim 2^j, \qquad j\geq {k_0}
\end{gather}

\bigskip 
The subsequent analysis depends upon the choice of $ 0 \leq s \leq j- k_0$.
We turn to the inequality \eqref{e:same}.
At the coarsest scale, the division of $ \mathcal I (j)$ is into 
$\bigcup_{s \geq 0} \mathcal S (j,s) \cup \mathcal N (j,s),$ where $ \mathcal N (j,s)$  consists of  those $ I\in \mathcal I (j)$ such that 
\begin{equation}\label{e:non}
| T_I^{\ast} T_I B _{j-s}| 
\leq 100 C _{\eqref{e:same}} \left( |I|^{-1} 1_{[-1/2,1/2]} * (B_{j-s} \mathbf{1}_{I'}) \right)
\end{equation}
Above, we use the implied constant $  C _{\eqref{e:same}}$ of \eqref{e:same}.  
The collection $ \mathcal S (j,s) $ are the intervals $I \in \mathcal{I}(j)$ so that \eqref{e:non} fails; these are the complementary, or `standard' collection, namely those intervals for which the second term on the right in \eqref{e:same} is dominant.

The standard collection is easy to analyze.  

\begin{lemma}\label{l:standard}We have the inequalities 
\begin{equation*}
\Bigl\lVert 
\sup _{j_0 \geq k_0} 
\Bigl\lvert \sum_{s=0} ^{\infty } 
\sum_{j = s+k_0} ^{\infty } T _{\mathcal S (j,s)} B _{j-s}
\Bigr\rvert\, 
\Bigr\rVert_q 
\lesssim   q \lvert  I_0\rvert ^{1/q}, \qquad 1 < q < \infty . 
\end{equation*}
 
\end{lemma}

\begin{proof}
We will see gain as the scale parameter $j$ increases.  
There is no cancellation needed.  Since $ \langle b \rangle_I \lesssim 1$ for all $ I\in \mathcal I$, we have for fixed $j$, 
\begin{align*}
\Bigl\lVert \sum_{s=0} ^{j-k_0 }  T _{\mathcal S (j,s)} B _{j-s} \Bigr\rVert _{\infty } \lesssim \sup _{I\in \mathcal{I}} \langle f\rangle_I \lesssim 1.  
\end{align*}
But, also, by construction of the standard collection, the second term in \eqref{e:same} dominates, hence 
\begin{align*}
\Bigl\lVert \sum_{s=0} ^{j-k_0 }    T _{\mathcal S (j,s)} B _{j-s} \Bigr\rVert  _{2 }  ^{2} 
& \lesssim
j  \sum_{s=0} ^{j-k_0 }   \lVert T _{\mathcal S (j,s)} B _{j-s} \rVert  _{2 }  ^{2}
\\& \lesssim j
\sum_{s=0} ^{j-k_0 } \sum_{I \in \mathcal{S}(j,s)} 
\vert I \rvert^{-2} \| B_{j-s} \mathbf{1}_I \|_1 ^2
\\ & 
\lesssim  j \sup_s  \sup_{I \in \mathcal{S}(j,s)} \vert I \rvert^{-2} \| B_{j-s} \mathbf{1}_I \|_1  
\sum_{s=0} ^{j-k_0 } \sum_{I \in \mathcal{S}(j,s)} 
 \| B_{j-s} \mathbf{1}_I \|_1 \lesssim 2 ^{-j/2} \lvert I_0\rvert . 
\end{align*}
The geometric decay comes from the length of $I$. 

Interpolating between these two bounds shows that 
\begin{equation*}
\Bigl\lVert \sum_{s=0} ^{j-k_0 }  T _{\mathcal S (j,s)} B _{j-s} \Bigr\rVert_q 
\lesssim 2 ^{-j/2q} \lvert  I_0\rvert  ^{1/q}, \qquad  2< q < \infty. 
\end{equation*}
Summing over $ j\geq k_0$   proves the estimate. 

\end{proof}

In the remainder of the argument, we hold $ s \geq 0$ fixed, and gain geometric decay in $ s$. 
A key estimate is the $ L ^{\infty }$ bound. 
\begin{proposition}\label{p:infty}
Assume that the function $g$ satisfies \eqref{e:K}. 
Uniformly in $ s\geq 0$, we have 
\begin{equation}\label{e:infty}
 \Bigl\langle 
 \sup _{j_0 \geq s+k_0} 
\Bigl\lvert  
\sum_{j = j_0} ^{\infty } T _{\mathcal N(j,s)} B _{j-s}\Bigr\rvert ,g 
\Bigr\rangle  \lesssim  \lvert I_0\rvert \langle f\rangle_{I_0} \langle g\rangle_{I_0}. 
\end{equation}
\end{proposition}
\begin{proof}
We linearize the maximal truncations via a measurable selection 
function $ \varepsilon (x)$, and set
\begin{equation*}
\tilde T _{I} f (x) =   \mathbf 1_{ \{\lvert  I\rvert > \varepsilon (x) \}} T_I f (x).  
\end{equation*}
Then, for any integrable function $ g$ on $ I_0$ we have 
\begin{align*}
\Bigl\langle  
\sum_{j = s+k_0} ^{\infty }  \sum_{I\in \mathcal N (j,s)} \tilde T _I B _{j-s}, g 
\Bigr\rangle
& = 
\sum_{j = s+k_0} ^{\infty }  \sum_{I\in \mathcal N (j,s)}  \langle  \tilde T _I  B _{j-s} ,   g\rangle 
\\
& \lesssim  \sum_{j = s+k_0} ^{\infty }  \sum_{I\in \mathcal N (j,s)}  \int _{I} B _{j-s} \; dx  \cdot 
\langle g\rangle_{I}
\lesssim   \lvert  I_0\rvert \langle f\rangle_{I_0}\langle g\rangle_{I_0}
\end{align*}
by \eqref{e:Bsum}.  This proves \eqref{e:infty}.  

\end{proof}

The principle estimate, indeed the core of the argument, concerns the $ L ^{2} $ estimate for the maximal truncations. 

\begin{lemma}\label{l:3prove} We have this estimate, uniformly in $ s \geq 0$.  
\begin{equation}\label{e:3prove}
\Bigl\lVert 
\sup _{j_0 \geq s+k_0} 
\Bigl\lvert  
\sum_{j = j_0} ^{\infty } T _{\mathcal N(j,s)} B _{j-s}
\Bigr\rvert\, 
\Bigr\rVert_2 
\lesssim    2 ^{-s/5}  \lvert  I_0\rvert ^{1/2}. 
\end{equation}

\end{lemma}

With this Lemma proved, interpolate between \eqref{e:infty} and \eqref{e:3prove} to see that 
\begin{equation*}
 \Bigl\langle 
 \sup _{j_0 \geq s+k_0} 
\Bigl\lvert  
\sum_{j = j_0} ^{\infty } T _{\mathcal N(j,s)} B _{j-s}\Bigr\rvert ,g 
\Bigr\rangle \lesssim    2 ^{-s/5q}  \lvert  I_0\rvert  
\langle f\rangle_{I_0}\langle g\rangle_{I_0,q'}, \qquad  2< q < \infty. 
\end{equation*}
A sum over $ s \geq 0$ contributes a  power of $ q$, completing the proof of \eqref{e:q*}.  

\begin{proof}

Here is the main claim.  
There is a subset $ F_s\subset I_0$ with $ \lvert  F_s\rvert \leq \tfrac 14 \lvert  I_0\rvert  $ so that setting 
$ \mathcal N ^{\sharp}(j,s) = \{I\in \mathcal N (j,s) \,:\, I\not\subset F_s\}$,  
\begin{gather}  \label{e:st2}
\Bigl\lVert 
\sup _{j_0 \geq s+k_0} 
\Bigl\lvert  
\sum_{j = j_0} ^{\infty } T _{\mathcal N  ^{\sharp} (j,s)} B _{j-s}
\Bigr\rvert\, 
\Bigr\rVert_2 
\lesssim   2 ^{-s/5} |I_0|^{1/2}. 
\end{gather}
This is the inequality we want, except that we have excluded the intervals $I\subset F_s$. 
But, we can recurse inside the set $F_s$, and since it has small measure relative to $I_0$, we can conclude that 
the unrestricted estimate below holds. 
\begin{gather}  
\Bigl\lVert 
\sup _{j_0 \geq s+k_0} 
\Bigl\lvert  
\sum_{j = j_0} ^{\infty } T _{\mathcal N   (j,s)} B _{j-s}
\Bigr\rvert\, 
\Bigr\rVert_2 
\lesssim   2 ^{-s/5} |I_0|^{1/2}. 
\end{gather}
Summing this over $t\geq 0$ concludes the estimate \eqref{e:3prove}.

To prove \eqref{e:st2}, we will the Rademacher-Menshov Lemma \ref{l:RM}, which requires that the intervals in $\mathcal{N}_{j,s}$ have bounded overlaps.  This is almost true, by this  Carleson measure estimate.  
\begin{equation}\label{e:CM}
\sum_{j = s+k_0} ^{\infty } \sum_{I\in \mathcal N (j,s) \,:\, I\subset J} \lvert  I\rvert \lesssim  2 ^{s} \lvert  J\rvert  , \qquad  J \in \mathcal I. 
\end{equation}
  Indeed, for each $I \in \mathcal N (j,s)$ there is one `bad' interval $K\in \mathcal{B}(j-s)$, so that 
  $\int _I B _{j-s} \geq \int _K B _{j-s} \gtrsim 2^{-s} \lvert I\rvert$. (See \eqref{e:B}.)  Therefore, by \eqref{e:Bsum}, 
\begin{align*}
\sum_{j = s+k_0} ^{\infty } \sum_{I\in \mathcal N (j,s,t) \,:\, I\subset J} \lvert  I\rvert
&\lesssim  2^{s} \sum_{j = s+k_0} ^{\infty } \sum_{I\in \mathcal N (j,s,t) \,:\, I\subset J} \lVert B _{j-s} \mathbf 1_{I}\rVert_1 
\lesssim  {2^{s}}\lVert b \mathbf 1_{J}\rVert_1 \lesssim  2^{s} \lvert  J\rvert.  
\end{align*}
This proves \eqref{e:CM}.

\medskip 

Take the set $ F_s$ of  \eqref{e:st2} to be 
\begin{equation}\label{e:Fs}
F_s = 
\Bigl\{
\sum_{j = s+k_d} ^{\infty } \sum_{I\in \mathcal N (j,s,t)} \mathbf 1_{I} >  C s 2 ^{s} 
\Bigr\}. 
\end{equation}
It follows from \eqref{e:CM}, the fact that each $I \in \mathcal{N}(j,s)$ is contained in $\{ M_{HL}f \gtrsim 2^{-s}\}$, and the John-Nirenberg inequality that $ \lvert  F_s\rvert \leq \frac{1}{4} \lvert  I_0\rvert$, provided $C$ is sufficiently large.  

We turn our attention to the maximal  $ L ^2 $ estimate, which will follow from the Rademacher-Menshov Lemma \ref{l:RM}.  To set up its application, this additional noation is required.  
Set $ u_0 =   C 2 ^{s} $. 
Define 
\begin{equation*}
\beta _j = \sum_{J\in \mathcal K_j} T _{J} b , \qquad 1\leq j \leq  u_0, 
\end{equation*}
where  $ \mathcal K_{j}$ is a generational decomposition of $ \mathcal N ^{\sharp} = \bigcup _{j=k_0} ^{\infty }  \mathcal  N  ^{\sharp}_{j,s,t}  $.  
Namely $ \mathcal K_1$ is the minimal elements of $  \mathcal  N  ^{\sharp}$, and $ \mathcal K_2$ is the minimal elements of 
$ \mathcal  N  ^{\sharp}\setminus \mathcal K_1$, and so on. The point of this choice is that we necessarily have $ \lvert  J\rvert \geq 2 ^{j} $ for $ J\in \mathcal K_j$.

This is the core of the argument.  
We show that for any choice of coefficients $ c_j \in \{-1, 0 ,1\}$, 
\begin{equation}\label{e:N3}
\Bigl\lVert   \sum_{j=k_0} ^{u_0}  c_j \beta _j   \Bigr\rVert_2 
\lesssim  s 2^{-s/2}  \lvert  I_0\rvert ^{1/2}  . 
\end{equation}
This verifies the assumption \eqref{e:bessel} of Rademacher-Menshov Lemma~\ref{l:RM}. By \eqref{e:RM}, we can control 
the maximal truncations at a cost of a factor of $ \log u_0  \lesssim  s  $.  
And, then  \eqref{e:st2} follows. 

\medskip

We square out the norm on the left in \eqref{e:N3}. On the one hand, we have from \eqref{e:same}, and the assumption that the interval is `non-standard', see \eqref{e:non},  that 
\begin{align} \label{e:D}
 \sum_{j=k_d} ^{u_0}  c_j ^2 \lVert \beta _j  \rVert_2 ^2 
 &= \sum_{J\in \mathcal  N  ^{\sharp}_{j,s,t}} \lVert T_J f\rVert_2 ^2 
 \lesssim  \sum_{  J \in \mathcal  N  ^{\sharp}_{j,s,t}} \lvert  J\rvert^{-1}  {\lVert  B _{j-s}  \mathbf{1}_{[-1/2,1/2]} \ast  (B _{j-s}\mathbf 1_{J'})\rVert_1  }  
\\
& \lesssim 
 \sum_{  J \in \mathcal  N  ^{\sharp}_{j,s,t}} 
  \lvert  J\rvert^{-1} {\lVert   \mathbf{1}_{[-1/2,1/2]} \ast  (B _{j-s}\mathbf 1_{J'})\rVert_ \infty }  \lVert  B _{j-s} \rVert_1 
  \\  \label{e:xjj}
  & \lesssim 2^{-s}
   \sum_{  J \in \mathcal  N  ^{\sharp}_{j,s,t}} 
    \lVert  B _{j-s} \rVert_1 
    \lesssim 2 ^{-s} \lvert I_0\rvert . 
\end{align}
We have appealed to \eqref{e:BI} to control the $L^\infty$ norm,  and  \eqref{e:Bsum} to control the last sum.

Now, for $ 1\leq k < j \leq u_0$, notice that we have from above, that 
\begin{equation} \label{e:jj}
\lvert  \langle \beta _j, \beta _k \rangle\rvert \lesssim \lVert \beta _j\rVert_2 \lVert \beta _k\rVert_2 
\lesssim  2 ^{-s}\lvert  I_0\rvert. 
\end{equation}
This is useful when $ j, k$ are relatively small.

Otherwise,  recalling that $ k_0$ is a fixed large integer, 
for $ k + k_0 < j$,  
\begin{align}
\lvert   \langle \beta _j , \beta _k \rangle \rvert &\leq 
\sum_{J\in \mathcal K_j}  
\sum_{K\in \mathcal K_k \,:\, K\subset J}\lvert   \langle T_K ^{\ast} T_J B _{j (J)-s} ,   B _{j (K) -s}  \mathbf 1_{K}   \rangle \rvert 
\end{align}
Above, we are using the notation $ \lvert  K\rvert = 2 ^{j (K)} $.  Then, again using the construction, 
and  the stronger orthogonality condition \eqref{e:different}, 
\begin{align*}
\lvert   \langle T_K ^{\ast} T_J B _{j (J)-s} ,   B _{j (J) -s}  \mathbf 1_{K'}   \rangle \rvert   
&\lesssim \lvert  J\rvert ^{-2} \lVert B _{j (J)-s}  \mathbf 1_{J} \rVert_1  \lVert B _{j (K) -s} \mathbf 1_{K}\rVert_1 
\\
& \lesssim  2 ^{-j}  \tfrac {\lvert  K\rvert } {\lvert  J\rvert } \lVert B _{j (J)-s}  \mathbf 1_{J} \rVert_1. 
\end{align*}
Here, we used that intervals $J \in \mathcal{J}_j$ must have length at least $ 2^{j}$.
Combining these estimates, we have 
\begin{align*}
\lvert  \langle \beta _j, \beta _k \rangle\rvert 
&\lesssim 
2 ^{- j }
\sum_{J\in \mathcal K_j}  
\sum_{K\in \mathcal K_k \,:\, K\subset J} 
 \tfrac {\lvert  K\rvert } {\lvert  J\rvert } \lVert B _{j (J)-s}  \mathbf 1_{J} \rVert_1
\\
& \lesssim 
2 ^{- j }
\sum_{J\in \mathcal K_j}   \lVert B _{j (J)-s}  \mathbf 1_{J} \rVert_1 \lesssim 2 ^{- j} \lvert  I_0\rvert.  
\end{align*}
Combining this estimate with \eqref{e:xjj} and \eqref{e:jj}, we conclude \eqref{e:N3}, completing the proof.

\end{proof}

\bibliographystyle{amsplain}

\end{document}